\documentclass[12pt]{amsart}
\usepackage{amsmath, fullpage,amsthm, stmaryrd,url}
\usepackage[all]{xy}

\newtheorem{theorem}{Theorem}[section]
\newtheorem{lemma}[theorem]{Lemma}
\newtheorem{cor}[theorem]{Corollary}
\newtheorem{prop}[theorem]{Proposition}

\theoremstyle{definition}

\newtheorem{hypothesis}[theorem]{Hypothesis}
\newtheorem{defn}[theorem]{Definition}
\newtheorem{example}[theorem]{Example}
\newtheorem{remark}[theorem]{Remark}
\numberwithin{equation}{theorem}

\newcommand{\FF}{\mathbb{F}}
\newcommand{\QQ}{\mathbb{Q}}
\newcommand{\RR}{\mathbb{R}}
\newcommand{\ZZ}{\mathbb{Z}}

\newcommand{\bv}{\mathbf{v}}
\newcommand{\bw}{\mathbf{w}}
\newcommand{\dual}{\vee}

\DeclareMathOperator{\aut}{aut}
\DeclareMathOperator{\End}{End}

\DeclareMathOperator{\Hom}{Hom}
\DeclareMathOperator{\id}{id}

\DeclareMathOperator{\len}{len}
\DeclareMathOperator{\ord}{ord}

\DeclareMathOperator{\Post}{Post}
\DeclareMathOperator{\post}{post}
\DeclareMathOperator{\Pre}{Pre}
\DeclareMathOperator{\pre}{pre}
\DeclareMathOperator{\rev}{rev}

\begin{document}

\title{On the algebraicity of generalized power series}
\author{Kiran S. Kedlaya}
\date{November 24, 2016}
\thanks{Thanks to Rodrigo Salom\~ao and Reillon Santos for initiating the discussion about \cite{k-series1} that led to the identification of Example~\ref{exa:not TR}.
Thanks also to Jason Bell, Harm Derksen, Anna Medvedovsky, and David Speyer for additional discussions. The author was supported by NSF (grants DMS-1101343, DMS-1501214) and UCSD (Stefan E. Warschawski chair).}

\begin{abstract}
Let $K$ be an algebraically closed field of characteristic $p$.
We exhibit a counterexample against a theorem asserted in one of our earlier papers, which claims to characterize the integral closure of $K((t))$ within the field of Hahn-Mal'cev-Neumann generalized power series. We then give a corrected characterization, generalizing our earlier description in terms of finite automata in the case where $K$ is the algebraic closure of a finite field.
We also characterize the integral closure of $K(t)$, thus generalizing a well-known theorem of Christol and suggesting a possible framework for computing in this integral closure. We recover various corollaries on the structure of algebraic generalized power series; one of these is an extension of Derksen's theorem on the zero sets of linear recurrent sequences in characteristic $p$.
\end{abstract}

\maketitle

\section{Introduction}

It is well-known that for $K$ an algebraically closed field of characteristic zero, an algebraic closure of the field $K((t))$ of formal Laurent series can be constructed by forming the field of \emph{Puiseux series}
\[
\bigcup_{n=1}^{\infty} K((t^{1/n})).
\]
For the remainder of this paper, let $K$ instead be an algebraically closed field of characteristic $p>0$. In this case, the field of Puiseux series is not algebraically closed, e.g., because it contains no element $x$ satisfying $x^p - x = t^{-1}$.

The goal of our previous paper \cite{k-series1} was to construct an algebraic closure of $K((t))$ within the field $K((t^{\QQ}))$ of \emph{generalized power series}. An element of this field (which may also be called a \emph{Hahn series} or \emph{Mal'cev-Neumann series}) is a formal linear combination $\sum_{i \in \QQ} c_i t^i$ with $c_i \in K$ whose support (i.e., the set of $i \in \QQ$ for which $c_i \neq 0$) is a well-ordered subset of $\QQ$.
The natural addition and multiplication operations turn out to be well-defined and give $K((t^\QQ))$ the structure of an algebraically closed field; consequently, the integral closure of $K((t))$ within this field is itself an algebraic closure of $K((t))$.

Unfortunately, it was recently pointed out by Rodrigo Salom\~ao and Reillon Santos that if $K \neq \overline{\FF}_p$, the main theorem of \cite{k-series1} is false as stated: the subset of $K((t^{\QQ}))$ identified in \cite[Theorem~8]{k-series1} is in fact a proper subset of the integral closure of $K((t))$. The first purpose of this paper is to describe a simple explicit counterexample to \cite[Theorem~8]{k-series1} (see Example~\ref{exa:not TR}) and use it to isolate the error in the arguments of \cite{k-series1}.
The second purpose of this paper is to give a corrected description of the integral closure of $K((t))$ in $K((t^{\QQ}))$ (see Theorem~\ref{T:main}), recover all of the other results of \cite{k-series1} (with minor modifications as needed), and verify that the statements of 
\cite[Theorem~8, Corollary~11]{k-series1} are correct when $K = \overline{\FF}_p$.

As in \cite{k-series1}, the description of the algebraic closure of $K((t))$ consists of two parts: a restriction on the supports of generalized power series, and a restriction on the coefficients corresponding to elements of the support. The restriction on the supports turns out to be the same as in \cite{k-series1}, but the restriction on coefficients is somewhat more permissive.

This modified description is inspired by our later paper \cite{k-series-automata},
in which the case where $K$ is the algebraic closure of a finite field $\FF_q$ is given an alternate treatment using \emph{finite automata} (or equivalently \emph{regular languages}), in the manner of Christol's theorem on the integral closure of $\FF_q(t)$ within $\FF_q((t))$. We recast this setup in terms of linear algebra (or more precisely, semilinear algebra with respect to Frobenius) to obtain a valid result for arbitrary $K$. In the process, we characterize the integral closures of both $K(t)$ and $K((t))$ in $K((t^{\QQ}))$ (see Theorem~\ref{T:integral closure of rational field} for the former) and thus in particular obtain an extension of Christol's theorem. We also obtain a framework for computing in these integral closures, which we expect to be more practical than the direct use of finite automata suggested in \cite{k-series-automata}; see \S\ref{sec:computational}
for further discussion.

As in \cite{k-series1}, our main theorems have a number of corollaries which can be stated somewhat more simply. Most of these are collected in \S\ref{sec:corollaries}, to which the first-time reader may wish to turn first. In \S\ref{sec:zero sets}, we deduce an extension of Derksen's positive-characteristic analogue of the Skolem-Mahler-Lech theorem \cite{derksen}, which asserts (in its simplest form) that the zero set of a linear recurrent sequence over $K$ is $p$-automatic; namely, we prove the analogous statement for the coefficients of any generalized power series which is algebraic over $K(t)$. Interestingly, the argument specializes to a short and apparently new proof of Derksen's theorem itself.
(It should be noted that in \cite{derksen}, one also finds some more refined statements about the structure of zero sets of linear recurrent sequences; we do not treat these statements here.)

We note in passing that in addition to the aforementioned papers, our paper \cite{k-series2} on the construction of algebraic closures of $p$-adic fields also makes use of \cite{k-series1}. The statements of \cite[Theorem~10, Theorem~11]{k-series2} must thus be corrected in a similar manner (see Theorem~\ref{T:mixed1} and Theorem~\ref{T:mixed2}).
At the end of the paper (see \S\ref{sec:effect}), we conduct a survey of published papers that cite \cite{k-series1, k-series2} to confirm that (as far as we have able to determine) no other results are affected.

\section{A counterexample}

We begin by introducing enough terminology and notation to recall the statement of the main theorem of \cite{k-series1}, then describe the counterexample alluded to in the introduction.

\begin{defn} \label{D:sabc}
For $a$ a positive integer, $b$ an integer, and $c$ a nonnegative integer, define the set
\[
S_{a,b,c} = \left\{\frac{1}{a} (n- b_1 p^{-1} - b_2 p^{-2} - \cdots): n \geq -b, b_i \in \{0,\dots,p-1\}, \sum b_i \leq c\right\}.
\]
In words, $S_{a,b,c}$ consists of those $s \in \QQ$ such that the fractional part of $-as$, when written in base $p$, has digit sum at most $c$.
\end{defn}

\begin{defn} \label{D:LRR}
For $k$ a nonnegative integer,
a \emph{linearized recurrence relation (LRR)} over $K$ of length $k$ is a condition on an infinite sequence $c_0, c_1, \dots$ taking the form
\begin{equation} \label{eq:LRR}
d_0 c_n + d_1 c_{n+1}^p + \cdots + d_k c_{n+k}^{p^k} = 0 \qquad (n=0,1,\dots)
\end{equation}
for some $d_0,\dots,d_k \in K$. Note that the set of solutions of a given LRR is stable under addition and $\FF_p$-scalar multiplication, but not $K$-scalar multiplication: multiplying a solution by a scalar in $K$ gives instead a solution of a new LRR of the same length.
\end{defn}

\begin{defn}
Let $T_c = S_{1,0,c} \cap (-1,0)$. For $k$ a positive integer, a function $f: T_c \to K$ is 
\emph{twist-recurrent of order $k$} if there exist $d_0,\dots,d_k \in K$ such that the LRR
\eqref{eq:LRR} holds for any sequence $\{c_n\}$ of the form
\begin{equation} \label{eq:TR sequences}
c_n = f(-b_1 p^{-1} - \cdots - b_{j-1} p^{-j+1}+ p^{-n}(b_j p^{-j} + \cdots)) \qquad (n\geq 0)
\end{equation}
for $j$ a positive integer and $b_1,b_2,\ldots \in \{0,\dots,p-1\}$ with $\sum b_i \leq c$.
\end{defn}

\begin{defn} \label{D:twist-recurrent}
A series $x \in K((t^\QQ))$ is \emph{twist-recurrent} if the following conditions hold.
\begin{enumerate}
\item[(a)]
There exist $a,b,c$ such that $x$ is supported on $S_{a,b,c}$.
\item[(b)]
For some (hence any) $a,b,c$ as in (a) and each integer $m \geq -b$, the function $f_m: T_c \to K$ given by $f_m(z) = x_{(m+z)/a}$ is twist-recurrent (of some order).
\item[(c)]
The functions $f_m$ span a finite-dimensional vector space over $K$.
\end{enumerate}
\end{defn}

The proof of \cite[Theorem~8]{k-series1} consists of three steps, the first two of which are correct and yield the following result. We will recover this statement later as a
consequence of Theorem~\ref{T:main}.
\begin{prop} \label{P:TR algebraic}
The twist-recurrent elements of $K((t^\QQ))$ form a subring which is integral over $K((t))$. In particular, this subring, being integral over a field, is itself a field.
\end{prop}

However, \cite[Theorem~8]{k-series1} also asserts that every element of $K((t^\QQ))$ which is integral over $K((t))$ is twist-recurrent, which is false in general; here is an explicit counterexample.
\begin{example} \label{exa:not TR}
Choose $\lambda \in K$ nonzero. Put
\[
y = \sum_{m=1}^\infty \lambda^{p^{-1}+\cdots+p^{-m}} t^{-p^{-1}-p^{-1-m}}, \qquad
x = \sum_{m,n=1}^\infty \lambda^{p^{-1-n}+\cdots+p^{-m-n}} 
t^{-p^{-1-n}-p^{-1-m-n}},
\]
so that $t y^p - \lambda t^{1/p} y = \lambda t^{-1/p}$, $x^p - x = y$.
Note that $x$ satisfies Definition~\ref{D:twist-recurrent}(a) for the parameters $(a,b,c) = (1,0,2)$. Write $x = \sum_i x_i t^i$ and let $f: T_c \to K$ be the function given by $f(z) = x_{z}$. 
For each $j \geq 2$, consider the sequence as in \eqref{eq:TR sequences} for
\[
b_1 = \cdots = b_{j-2} = 0, \quad b_{j-1} = b_j = 1, \quad b_{j+1} = b_{j+2} =  \cdots = 0;
\]
this sequence satisfies
\[
c_n = \lambda^{p^{1-j}+\cdots+p^{1-j-n}}. 
\]
This sequence satisfies the LRR \eqref{eq:LRR} if and only if
\[
d_0 + d_1 \lambda^{p^{2-j}} + d_2 (\lambda^{p^{2-j}})^{1+p} + \cdots + d_k (\lambda^{p^{2-j}})^{1+p+\cdots+p^{k-1}} = 0.
\]
However, if $\lambda$ is not integral over $\FF_p$, then $\lambda^{p^{2-j}}$ takes infinitely many values, only finitely many of which can be roots of the polynomial 
\[
P(T) = d_0 + d_1 T + d_2 T^{1+p} + \dots + d_k T^{1+p+\cdots+p^{k-1}}.
\]
Consequently, $f$ is not twist-recurrent, and thus $x$ cannot be twist-recurrent.
\end{example}

We now isolate the error in the proof of \cite[Theorem~8]{k-series1} exposed by Example~\ref{exa:not TR}.
\begin{remark} \label{R:not TR}
The third step in the proof of  \cite[Theorem~8]{k-series1} is to show that if
$x = \sum_i x_i t^i,y = \sum_i y_i t^i \in K((t^{\QQ}))$ are such that $x^p - x = y$ and $y$ is twist-recurrent, then $x$ is twist-recurrent. It suffices to treat the cases where $y$ is supported either on $(-\infty, 0) \cap S_{a,b,c}$ or on $(0, \infty) \cap S_{a,b,c}$ for some $a,b,c$.

In case $y$ is supported on $(-\infty, 0) \cap S_{a,b,c}$, then
\[
x = \sum_i x_i t^i = \sum_i \sum_{n=1}^\infty y_i^{1/p^n} t^{i/p^n} = \sum_i t^i \sum_n y_{ip^n}^{1/p^n}
\]
is supported on $S_{a,b,b+c}$. 
We must show that there exists a single LRR which is satisfied by every sequence of the form
\[
c_n = x_{m-b_1 p^{-1} - \cdots - b_{j-1} p^{-(j-1)} - p^{-n}(b_j p^{-j} + \cdots)}
\qquad (n=0,1,\dots)
\]
for $j \geq 0$, $-b \leq m \leq 0$, $b_i \in \{0,\dots,p-1\}$, and $\sum_i b_i \leq c$. 
If $m=j=0$, then
\[
c_{n+1}^p - c_n = y_{-b_1p^{-1} -\cdots - b_{j-1} p^{-(j-1)} - p^{-n}(b_j p^{-j} + \cdots)};
\]
if $\{c_{n+1}^p - c_n\}$ satisfies an LRR with coefficients $d_0,\dots,d_k$, then $\{c_n\}$ satisfies an LRR with coefficients $-d_0, d_0-d_1,\dots,d_{k-1}-d_{k}$. (The last coefficient is misprinted as $d_k - d_{k-1}$ in \cite{k-series1}.)
By contrast, if $m<0$ or $j>0$, then $\{c_n\}$ is the sum of a bounded number of sequences,
each obtained from a sequence satisfying a fixed LRR by taking $p^k$-th roots for some $k$.
One then needs to apply \cite[Corollary~5]{k-series1} to deduce that $\{c_n\}$ satisfies a fixed LRR, but this is only valid if the possible values of $k$ are uniformly bounded over all possible sequences, and this is precisely what fails to occur in Example~\ref{exa:not TR}.
\end{remark}

\begin{remark} \label{R:not TR2}
Although this is not made apparent by Example~\ref{exa:not TR},
the third step in the proof of  \cite[Theorem~8]{k-series1} 
contains a similar error in the case where $y$ is supported on $(0, \infty) \cap S_{a,b,c}$. In addition, the argument given is not sufficient to confirm
that $x$ satisfies condition (c) of Definition~\ref{D:twist-recurrent}.
\end{remark}

\begin{remark}
The statement of \cite[Theorem~8]{k-series1} is correct as written in the case $K = \overline{\FF}_p$. This can be seen either by inspecting the proof carefully to see that
the errors described above do not affect this case, or by appealing to those results of
\cite{k-series-automata} which are derived independently of \cite{k-series1}
(see Remark~\ref{R:fix automata}).
In this paper, we instead deduce \cite[Theorem~8]{k-series1} in the case $K = \overline{\FF}_p$ as a corollary of Theorem~\ref{T:main}, which is valid for all $K$; see Theorem~\ref{T:finite field}.
\end{remark}

\section{Regular languages and functions}

We next briefly introduce the related concepts of \emph{finite automata} and \emph{regular languages}. These concepts are both used in \cite{k-series-automata} to study generalized power series over the algebraic closure of a finite field, although they play a somewhat less central role here.

\begin{defn}
For $\Sigma$ a set (which we regard as an \emph{alphabet}), let $\Sigma^*$ be the set of finite sequences taking values in $\Sigma$ (which we regard as \emph{words} on the alphabet $\Sigma$). We view $\Sigma^*$ as a monoid for the concatenation operation, which we denote by juxtaposition (i.e., $xy$ is the concatenation of $x$ and $y$).
Let $\len: \Sigma^* \to \ZZ$ be the length function, so that
\[
\len(a_1 \cdots a_n) = n \qquad (a_1,\dots,a_n \in \Sigma).
\]
Let $\rev: \Sigma^* \to \Sigma^*$ be the reversal operator on strings, so that
\[
\rev(a_1\cdots a_n) = a_n \cdots a_1 \qquad (a_1,\dots,a_n \in \Sigma).
\]
\end{defn}

\begin{defn}
Let $\Sigma$ be a finite set and let $f: \Sigma^* \to \Delta$ be any function.
Define the \emph{prefix equivalence relation} $\sim_{f,\pre}$ and the \emph{postfix equivalence relation} $\sim_{f,\post}$ on $\Sigma^*$ by the formulas
\begin{gather*}
w \sim_{f,\pre} w' \Longleftrightarrow f(wx) = f(w'x) \mbox{ for all $x \in \Sigma^*$};  \\
w \sim_{f,\post} w' \Longleftrightarrow f(xw) = f(xw') \mbox{ for all $x \in \Sigma^*$}.\end{gather*}
Let $\Pre_f$ and $\Post_f$ be the sets of equivalence classes under $\sim_{f,\pre}$
and $\sim_{f,\post}$. We get well-defined concatenation maps
$\Pre_f \times \Sigma^* \to \Pre_f$, $\Sigma^* \times \Post_f \to \Post_f$.
Moreover, $f$ induces a map $f: \Pre_f \times \Post_f \to \Delta$.

We say that $f$ is a \emph{finite-state function} if $\Pre_f$ is finite; this forces $f$ to have finite image. By Lemma~\ref{L:rev} below, it is equivalent to require $\Post_f$ to be finite.
\end{defn}

\begin{defn} \label{D:regular language}
For $L \subseteq \Sigma^*$ (viewed as a \emph{language} on the alphabet $\Sigma$),
we say that $L$ is \emph{regular} if its characteristic function is a finite-state function. 

For $P$ a finite partition of a regular language $L$, we say that $P$ is \emph{regular} if 
the projection map $f_P: \Sigma^* \to P \cup \{\Sigma^* \setminus L\}$ is a finite-state function;
this happens if and only if each part of $P$ is itself a regular language
\cite[Theorems~4.3.1 and~4.3.2]{allouche-shallit}.
In this case, we write $\Pre_P, \Post_P$ instead of $\Pre_{f_P}, \Post_{f_P}$;
these are themselves regular partitions of $\Sigma^*$.

%
We formally extend the operator $\rev$ on strings to languages and regular partitions:
\[
\rev(L) = \{\rev(s): s \in L\}, \qquad \rev(P) = \{\rev(q): q \in P\}.
\]
By Lemma~\ref{L:rev} below, $\rev(P)$ is a regular partition if $P$ is;
note that $\rev(\Pre_P) = \Post_{\rev(P)}$ and $\rev(\Post_P) = \Pre_{\rev(P)}$.

\end{defn}

\begin{lemma} \label{L:rev}
Let $\Sigma$ be a finite alphabet.
\begin{enumerate}
\item[(a)]
A language $L \subseteq \Sigma^*$ is regular if and only if $\rev(L)$ is.
\item[(b)]
A function $f: \Sigma^* \to \Delta$ is a finite-state function  if and only
if $f \circ \rev$ is.
\end{enumerate}
\end{lemma}
\begin{proof}
Part (a) is \cite[Corollary~4.3.5]{allouche-shallit}; part (b) follows from (a)
using the criterion of Definition~\ref{D:regular language}.
\end{proof}

We next recall the basic type of finite automata considered in \cite[Definition~2.1.8]{k-series-automata}.
\begin{defn}
A \emph{deterministic finite automaton with output (DFAO)}
is a tuple $M = (Q, \Sigma, \delta, q_0, \Delta, \tau)$, where
\begin{itemize}
\item $Q$ is a finite set (the \emph{states});
\item $\Sigma$ is a finite set (the \emph{alphabet});
\item $\delta$ is a function from $Q \times \Sigma$ to $Q$ 
(the \emph{transition function});
\item $q_0 \in Q$ is a state (the \emph{initial state});
\item $\Delta$ is a set (the \emph{output alphabet});
\item $\tau$ is a function from $Q$ to $\Delta$ (the \emph{output function}).
\end{itemize}
The function $\delta$ formally extends to a function $\delta^*: 
Q \times \Sigma^*
\to Q$ by the rules 
\[
\delta^*(q, \emptyset) = q, \quad
\delta^*(q, wa) = \delta(\delta^*(q,w), a) 
\qquad (q \in Q, w \in \Sigma^*, a \in \Sigma).
\]
We then obtain a function $f_M: \Sigma^* \to \Delta$ by setting $f_M(w) = \tau(\delta^*(q_0, w))$.
\end{defn}

The relationship between finite-state functions and finite automata is provided by the Myhill-Nerode theorem.
\begin{prop}[Myhill-Nerode theorem] \label{P:Myhill-Nerode}
For $\Sigma$ a finite set, a function $f: \Sigma^* \to \Delta$ is a finite-state function
if and only if it occurs as $f_M$ for some DFAO $M$.
\end{prop}
\begin{proof}
Note that $\delta^*$ induces an injective map $\Pre_{f_M} \to Q$, so $f_M$ is a finite-state function; conversely, if $f$ is a finite-state function, there is a DFAO $M$ with $\Pre_{f_M} = Q = \Pre_f$. For more details, see \cite[Theorem~4.1.8]{allouche-shallit}.
\end{proof}

The following definition captures certain basic transformations of regular languages;
see Example~\ref{exa:p-adic arithmetic} for a key example.
\begin{defn}
Let $L$ be a regular language on the alphabet $\Sigma$. Let $P$ be a regular partition of $L$.
For $\tau: \Pre_P \times \Sigma \to \Delta$
a function, we define the function $f_\tau: \Sigma^* \to \Delta^*$ by the formula
\[
f_\tau(\emptyset) = \emptyset,\quad  f_\tau(wa) = f_\tau(w) \tau(w, a) \qquad (w \in \Sigma^*, a \in \Sigma).
\]
A function occurring as the restriction of $f_\tau$ to $L$ for some $P, \tau$ is called a \emph{uniform transducer} over $L$. 

For example, any map from $\Sigma$ to another alphabet $\Delta$ defines (using the trivial partition) a uniform transducer over $\Sigma^*$ whose image is again a regular language \cite[Theorem~4.3.6]{allouche-shallit}; such a map is called a \emph{homomorphism} of regular languages.
\end{defn}

\begin{lemma} \label{L:quantifier elimination}
Let $\Sigma, \Sigma'$ be two finite alphabets.
Let $\pi: \Sigma \times \Sigma' \to \Sigma$ be the projection map.
Let $f_\pi: (\Sigma \times \Sigma')^* \to \Sigma^*$ be the associated uniform transducer.
Then for any regular language $L$ on the alphabet $\Sigma \times \Sigma'$,
\[
\{s \in \Sigma^*: f_\pi^{-1}(\{s\}) \subseteq L\}
\]
is a regular language on the alphabet $\Sigma$.
\end{lemma}
\begin{proof}
Let $L'$ be the complement of $L$; then $L'$ is again a regular language. 
By the previous discussion, the image $f_\pi(L')$ is a regular language;
as then is its complement. However, the latter is exactly the set we are considering.
\end{proof}

%

\section{Regular languages and generalized power series}

We now make the link between regular languages and generalized power series, extending the case $K = \overline{\FF}_p$ treated in \cite{k-series-automata}. We begin by fixing some conventions regarding base-$p$ representations of rational numbers which are slightly nonstandard, but consistent with \cite{k-series-automata}.
\begin{defn}
Let $\Sigma_p$ denote the alphabet $\{0,\dots,p-1\}$.
Let $L_p^0 \subset \Sigma_p^*$ be the set of strings not beginning with $0$;
this is evidently a regular language.
For $s = s_0 \cdots s_n \in L_p^0$, define
\[
\left| s \right| = \sum_{j=0}^{n} s_j p^{n-j};
\]
this defines a bijection $\left| \bullet \right| : L_p^0 \to \ZZ_{\geq 0}$.
We view $s$ as our preferred base-$p$ representation of the nonnegative integer $\left| s \right|$; in particular, we represent $0$ by the empty string.

Let $L_p \subset (\Sigma_p \cup \{.\})^*$ be the set of the strings of the form $s_1.s_2$ with $s_1 \in L_p^0, s_2 \in \rev(L_p^0)$; this is again a regular language.
For $s = s_0 \cdots s_n \in L_p$, let $i$ be the unique index for which $s_i = .$, and define
\[
\left\| s \right\| = \sum_{j=0}^{i-1} s_j p^{i-j-1} + \sum_{j=i+1}^n s_j p^{i-j};
\]
this defines a bijection $\left\|\bullet \right\|$ from $L_p$ to $\ZZ[p^{-1}]_{\geq 0}$ (the set of nonnegative rational numbers with $p$-power denominators). We view $s$ as  our preferred base-$p$ representation of $\left\| \bullet \right\|$; beware that for nonnegative integers, we prefer different representations depending on whether we are viewing them as elements of $\ZZ_{\geq 0}$ (in which case we omit the trailing radix point) or $\ZZ[p^{-1}]_{\geq 0}$ (in which case we include the radix point).
\end{defn}

\begin{example} \label{exa:Dedekind cut}
For $r \in \RR$, the set of $s \in L_p$ such that $\left\|s\right\| < r$ is regular if and only if $r \in \QQ$.
\end{example}

\begin{example} \label{exa:p-adic arithmetic}
For $a \in \ZZ_{>0}$, $b \in \ZZ_{\geq 0}$, the function $i \mapsto ai+b$ on $\ZZ[p^{-1}]_{\geq 0}$ can be computed on base-$p$ expansions from right to left with only a fixed finite amount of auxiliary memory, but not from left to right. In other words, the function
\[
f: L_p \to L_p,
\qquad
\left\| \rev(f(s)) \right\| = a \left\| \rev(s) \right\| + b \qquad (s \in L_p)
\]
is a uniform transducer, but $\rev \circ f \circ \rev$ is not.
\end{example}

In this language, the theorem of Christol is the following. (For a rederivation of this theorem using the results of this paper, see Remark~\ref{R:recover Christol}.)
\begin{theorem}[Christol] \label{T:christol}
For $q$ a power of $p$, an element $x = \sum_{i=0}^\infty x_i t^i \in \FF_q \llbracket t \rrbracket$ is algebraic over $\FF_q(t)$ if and only the function $L_p^0 \to \FF_q$
taking $s$ to $x_{|s|}$ is a finite-state function.
\end{theorem}
\begin{proof}
See \cite[Th\'eor\`eme~1]{christol}, \cite[Th\'eor\`eme~1]{ckmr}, or \cite[Theorem~12.2.5]{allouche-shallit}.
\end{proof}

This theorem is generalized by \cite[Theorem~4.1.3]{k-series-automata} as follows.
For a rederivation of this result using the results of this paper, see Remark~\ref{R:recover automata1}.
\begin{theorem} \label{T:automata1}
For $q$ a power of $p$, an element $x = \sum_{i} x_i t^i \in \FF_q ((t^{\QQ}))$
is algebraic over $\FF_q(t)$ if and only it satisfies the following conditions.
\begin{enumerate}
\item[(a)]
For some integers $a,b$ with $a>0$, the support of $x$ is contained in $\{(i-b)/a: i \in \ZZ[p^{-1}]_{\geq 0}\}$.
\item[(b)]
For some $a,b$ as in (a), the function $L_p \to \FF_q$ taking $s$ to
 $x_{a\|s\|+b}$ is a finite-state function. (This then holds for any such $a,b$.)
\end{enumerate}
\end{theorem}

\begin{remark} \label{R:bad support}
One important distinction between Theorem~\ref{T:christol} and Theorem~\ref{T:automata1}
is that whereas any finite-state function $L_p^0 \to \FF_q$ defines a power series,  a finite-state function $L_p \to \FF_q$ typically does not define a generalized power series: the constraint of well-ordered support imposes a strong restriction on the underlying DFAO. We use one explicit consequence of this restriction in Lemma~\ref{L:accumulation point} below; see \cite[\S 7.1]{k-series-automata} for further discussion.
\end{remark}

\begin{lemma} \label{L:accumulation point}
Let $L \subseteq L_p$ be a regular language such that 
$\|L\| = \{\|s\|: s \in L\}$ is well-ordered. Then every accumulation point of $\|L\|$ belongs to $\QQ$.
\end{lemma}
\begin{proof}
Identify the characteristic function of $L$ with $f_M: (\Sigma_p \cup \{.\})^* \to \{0,1\}$ for some DFAO $M$.
Let $s_1, s_2,\ldots \in L$ be a sequence such that the sequence $\|s_1\|, \|s_2\|, \dots$ converges to some $r \in \RR$. Then there is a unique infinite string $s$ such that the length of the maximal prefix shared by $s$ and $s_i$ tends to $\infty$ as $i \to \infty$; this string may be viewed as a (possibly nonstandard) base-$p$ expansion of $r$.
Since $M$ has only finitely many states, there must exist a state $q$ which occurs as $\delta^*(q_0, p)$ for infinitely many prefixes $p$ of $s$. However, since $\|L\|$ is well-ordered, we may apply \cite[Theorem~7.1.6]{k-series-automata} to see that the transition graph of $M$ contains a unique minimal directed cycle from $q$ to itself. It follows that $s$ is eventually periodic, so $r \in \QQ$.
\end{proof}

\begin{remark} \label{R:fix automata}
Since \cite{k-series-automata} contains arguments based on \cite[Theorem~8]{k-series1} which is now known to be incorrect, some discussion is in order regarding the status of the results of \cite{k-series-automata}. The fact that generalized power series described by automata are algebraic over $\FF_q(t)$ is established by a direct calculation
\cite[Proposition~5.1.2]{k-series-automata} which does not refer to \cite{k-series1}.
The converse implication is given two distinct proofs: one of these \cite[Proposition~5.2.7]{k-series-automata} uses \cite[Theorem~15]{k-series1}, which we will later verify is correct as stated (see Theorem~\ref{T:finite field2}); the other
\cite[Proposition~7.3.4]{k-series-automata} is logically independent of \cite{k-series1}.
Consequently, all of the results of \cite{k-series-automata} remain correct as stated and proved.
\end{remark}

\section{Composed functions over a general field}
\label{sec:composed functions}

In order to extend the notion of automaticity to arbitrary $K$, we need to incorporate some additional linear algebra into the construction; this linear algebra in some sense makes the finite automata redundant, but for various reasons (including but not limited to compatibility with \cite{k-series-automata}) it is useful to keep them present.

\begin{hypothesis}
Throughout \S\ref{sec:composed functions}, let $F$ be a field (of any characteristic), and let $\varphi: F \to F$ be an automorphism.
\end{hypothesis}

\begin{defn}
Throughout this definition, let $V$ be a finite-dimensional $F$-vector space.
Let $V^\dual = \Hom_F(V,F)$ denote the dual vector space and let $\langle \bullet, \bullet \rangle: V^\dual \times V \to F$ denote the contraction (evaluation) pairing.
Let $\End_{\ZZ}(V)$ be the set of endomorphisms of the underlying abelian group of $V$,
viewed as a monoid with respect to composition of functions (so that $g \circ f$ means apply $f$, then $g$).

We say that $g \in \End_{\ZZ}(V)$ is \emph{$\varphi$-semilinear} if for all $r \in F, \bv \in V$, we have $g(r\bv) = \varphi(r) g(\bv)$.
Let $\End_\varphi(V)$ be the subset (not a submonoid under composition) of $\End_{\ZZ}(V)$ consisting of $\varphi$-semilinear endomorphisms.
Let $T: \End_{\varphi}(V) \to \End_{\varphi^{-1}}(V^\dual)$ denote the transpose map, characterized by the identity 
\[
\langle T(f)(\bv^\dual), \bv \rangle = \varphi^{-1}(\langle \bv^\dual, f(\bv) \rangle) \qquad
(\bv \in V, \bv^\dual \in V^\dual, f \in \End_{\varphi}(V)).
\]
\end{defn}

\begin{defn}
For $G$ a monoid, let $\prod: G^* \to G$ denote the multiplication map on strings,
with $\prod \emptyset$ being the identity element of $G$.
\end{defn}

\begin{defn} \label{D:automatic1}
Let $V$ be a finite-dimensional vector space over $F$.
A function $f: \Sigma_p^* \to \End_{\ZZ}(V)$ is \emph{$\varphi$-composed} 
if there exists a function $\tau: \{0,\dots,p-1\} \to \End_{\varphi}(V)$ 
such that
\[
f(s) = \tau(a_1) \circ \cdots \circ \tau(a_n) \qquad (s = a_1 \cdots a_n \in \Sigma_p^*).
\]
In particular, $f(\emptyset) = \id_V$ and $f(s) \in \End_{\varphi^{\len(s)}}(V)$ for all $s \in \Sigma_p^*$.

More generally, we say that $f$ is \emph{$\varphi$-autocomposed} if there exist a regular partition $P$ of $\Sigma_p^*$ and a function $\tau: \Pre_P \times \{0,\dots,p-1\} \to \End_\varphi(V)$ such that
\[
f(s) = \prod f_\tau(s) \qquad (s \in \Sigma_p^*).
\]
In particular, any $\varphi$-composed function is $\varphi$-autocomposed with $P$ being the trivial partition of $\Sigma_p^*$.
\end{defn}

This last observation has a partial converse.
\begin{defn} \label{D:potentially composed}
Let $V$ be a finite-dimensional $F$-vector space. A function $f: \Sigma_p^* \to \End_{\ZZ}(V)$ is \emph{potentially $\varphi$-composed} (resp.\ \emph{potentially $\varphi$-autocomposed}) if there exist a finite-dimensional $F$-vector space $V'$, an $F$-linear injection $\iota: V \to V'$, an $F$-linear surjection $\pi: V' \to V$, and a $\varphi$-composed (resp.\ $\varphi$-autocomposed) function $f': \Sigma_p^* \to \End_{\ZZ}(V')$ such that
\begin{equation} \label{eq:potentially composed}
f(s) = \pi \circ f'(s) \circ \iota \qquad (s \in \Sigma_p^*).
\end{equation}
\end{defn}

\begin{lemma} \label{L:flatten composed}
Any potentially $\varphi$-autocomposed function is potentially $\varphi$-composed.
\end{lemma}
\begin{proof}
It is enough to check that any $\varphi$-autocomposed function $f$ as in Definition~\ref{D:automatic1}
is potentially $\varphi$-composed.
Let $V'$ be the direct sum of copies of $V$ indexed by $\Pre_P$.
Let $\iota: V \to V'$ be the map taking $\bv$ to $(\bv)_{q \in \Pre_P}$. Let $\pi: V' \to V$ be the projection onto the copy of $V$ indexed by the prefix class of the empty string.
Let $\tau': \{0,\dots,p-1\}
\to \End_\varphi(V')$ be the function taking $a$ to the map
\[
(\bv_q)_{q \in \Pre_P} \mapsto (\tau(q, a)(\bv_{qa}))_{q \in \Pre_P}.
\]
Let $f': \Sigma_p^* \to \End_{\ZZ}(V')$ be the $\varphi$-composed function associated to $\tau'$.
Then \eqref{eq:potentially composed} holds, as desired.
\end{proof}
\begin{cor} \label{C:automatic reverse}
Let $V$ be a finite-dimensional $F$-vector space.
Let $f: \Sigma_p^* \to \End_{\ZZ}(V)$ be a potentially $\varphi$-composed (resp.\ potentially $\varphi$-autocomposed) function.
Then the function $f': \Sigma_p^* \to \End_{\ZZ}(V^\dual)$ taking $s$ to
$T(f(\rev(s)))$ is potentially $\varphi^{-1}$-composed (resp.\ potentially $\varphi^{-1}$-autocomposed).
\end{cor}
\begin{proof}
The $\varphi$-composed case is trivial; the $\varphi$-autocomposed case then follows by Lemma~\ref{L:flatten composed}.
\end{proof}

\begin{remark} \label{R:tensor product}
For $V, W$ two finite-dimensional $F$-vector spaces,
$f: \Sigma_p^* \to \End_{\ZZ}(V)$ a (potentially) $\varphi$-autocomposed function,
and $g: \Sigma_p^* \to \End_{\ZZ}(W)$
a (potentially) $\varphi$-autocomposed function, the function
$\Sigma_p^* \to \End_{\ZZ}(V \otimes_F W)$ taking
$s$ to the endomorphism
$\bv \otimes \bw \mapsto f(s)(\bv) \otimes g(s)(\bw)$ is again (potentially) $\varphi$-autocomposed.
\end{remark}

\begin{defn} \label{D:restricted automatic}
For $V$ a finite-dimensional $F$-vector space and  $f: L_p^0 \to \End_{\ZZ}(V)$ a function, we say that $f$ is \emph{$\varphi$-automatic} if the function $\tilde{f}: \Sigma_p^* \to \End_{\ZZ}(V)$ taking $s$ to $f(s)$ if $s \in L_p^0$ and 0 otherwise is potentially $\varphi$-composed (or equivalently by Lemma~\ref{L:flatten composed}, potentially $\varphi$-autocomposed). Using Remark~\ref{R:tensor product}, one may verify that the restriction to $L_p^0$ of any potentially $\varphi$-autocomposed function on $\Sigma_p^*$ is $\varphi$-automatic.
\end{defn}

For generalized power series, we need an analogue of $\varphi$-automatic functions in which $L_p^0$ is replaced by $L_p$.
\begin{defn} \label{D:biautomatic}
Let $V$ be a finite-dimensional $F$-vector space. A function $f: L_p \to \End_{\ZZ}(V)$ is 
\emph{$\varphi$-biautomatic} if there exist a finite-dimensional $F$-vector space $V'$, an
$F$-linear injection $\iota: V \to V'$, an $F$-linear surjection $\pi: V' \to V$,
 a $\varphi$-composed function $f_1: \Sigma_p^* \to \End_{\ZZ}(V')$, and a $\varphi^{-1}$-composed function $f_2:\Sigma_p^* \to \End_{\ZZ}(V')$ such that
\[
f(s_1.s_2) = \pi \circ f_1(\rev(s_1)) \circ f_2(s_2) \circ \iota \qquad (s_1 \in L_p^0, s_2 \in \rev(L_p^0)).
\]
By Lemma~\ref{L:flatten composed}, one gets the same class of functions if one allows $f_1$ (resp.\ $f_2$) to be $\varphi$-autocomposed (resp.\ $\varphi^{-1}$-autocomposed).
By Corollary~\ref{C:automatic reverse}, $f$ is $\varphi$-biautomatic if and only if the function $f': L_p \to \End_{\ZZ}(V^\dual)$ taking $s$ to $T(f(\rev(s)))$ is.
\end{defn}

\begin{lemma} \label{L:composed affine}
Let $V$ be a finite-dimensional $F$-vector space.
Choose $a \in \ZZ_{>0}$, $b \in \ZZ_{\geq 0}$.
\begin{enumerate} 
\item[(a)]
Let $f: L_p^0 \to \End_{\ZZ}(V)$ be a function.
 Define the function 
$f': L_p^0 \to \End_{\ZZ}(V)$ as follows: for $s \in L_p^0$, if $\left| s \right| = a \left| t \right| + b$ for some $t \in L_p^0$, put $f'(s) = f(t)$; otherwise, put $f'(s) = 0$. Then $f$ is $\varphi$-automatic (as in Definition~\ref{D:restricted automatic}) if and only if $f'$ is.
\item[(b)]
Let $f: L_p \to \End_{\ZZ}(V)$ be a function. Define the function 
$f': L_p \to \End_{\ZZ}(V)$ as follows: for $s \in L_p$, if $\left\| s \right\| = a \left\| t \right\| + b$ for some $t \in L_p$, put $f'(s) = f(t)$; otherwise, put $f'(s) = 0$. Then $f$ is $\varphi$-biautomatic if and only if $f'$ is.
\end{enumerate}
\end{lemma}
\begin{proof}
By Corollary~\ref{C:automatic reverse}, we may check (a) by verifying the analogous statement with the strings reversed; this follows from 
Example~\ref{exa:p-adic arithmetic}.
 From (a) and Corollary~\ref{C:automatic reverse}, we easily deduce (b).
\end{proof}

\section{Automatic series over a general field}

We now specialize to the case of an algebraically closed field of characteristic $p$,
and put forward definitions of \emph{automaticity} for ordinary and generalized power series that extend the definitions over finite fields.

\begin{defn}
For $F$ a field of characteristic $p$, let $\varphi: F \to F$ be the Frobenius endomorphism $x \mapsto x^p$. In particular, we are assuming globally that $K$ is an algebraically closed field of characteristic $p$, so $\varphi$ acts as an automorphism on both $K$ and $K((t^{\QQ}))$.
\end{defn}

\begin{defn} \label{D:automatic series}
A power series $x = \sum_{n=0}^\infty x_n t^n \in K \llbracket t \rrbracket$ is \emph{$p$-automatic} if there exists a $\varphi$-automatic function $f: L_p^0 \to \End_{\ZZ}(K)$
(viewing $K$ as a vector space over itself) such that
\[
x_{|s|} = f(\rev(s))(1) \qquad (s \in L_p^0).
\]
Similarly, a generalized power series $x = \sum_i x_i t^i \in K((t^{\QQ}))$ is \emph{$p$-automatic} if its support is contained in $\ZZ[p^{-1}]_{\geq 0}$ and 
there exists a $\varphi$-biautomatic function $f: L_p \to \End_{\ZZ}(K)$ such that
\[
x_{\|s\|} = f(s)(1) \qquad (s \in L_p);
\]
this agrees with the previous definition in case $x \in K \llbracket t \rrbracket$.
\end{defn}

\begin{defn} \label{D:p-quasi-automatic}
We say $x= \sum_i x_i t^i \in K  \llbracket t^{\QQ} \rrbracket$ is \emph{$p$-quasi-automatic} if the following conditions hold.
\begin{enumerate}
\item[(a)]
For some integers $a,b$ with $a>0, b \geq 0$, the support of $x$ is contained in $\{(i-b)/a: i \in \ZZ[p^{-1}]_{\geq 0}\}$.
\item[(b)]
For some $a,b$ as in (a), the  series $\sum_i x_i t^{ai+b}$ is $p$-automatic. It will follow from Lemma~\ref{L:any ab} that the same will hold for any $a,b$ as in (a); in particular, if $x$ is supported on $\ZZ[p^{-1}]_{\geq 0}$, then $x$ is $p$-quasi-automatic if and only if $x$ is $p$-automatic.
\end{enumerate}
Let $K (( t^{\QQ} ))_{\aut}$ be the subset of $K (( t^{\QQ} ))$ consisting of $p$-quasi-automatic elements.
\end{defn}

\begin{lemma} \label{L:any ab}
For $x = \sum_i x_i t^i\in K((t^{\QQ}))$ supported on $\ZZ[p^{-1}]_{\geq 0}$,
and $a,b$ integers with $a > 0$, $b \geq 0$, $x$ is $p$-automatic if and only if $\sum_i x_i t^{ai+b}$ is. 
\end{lemma}
\begin{proof}
This is a special case of Lemma~\ref{L:composed affine}(b).
\end{proof}

The preceding definitions are compatible with \cite{k-series-automata} in the following sense.
\begin{lemma} \label{L:compatibility}
If $K = \overline{\FF}_p$, then $x \in K (( t^{\QQ} ))$ is $p$-automatic (resp.\ $p$-quasi-automatic) in the sense of Definition~\ref{D:p-quasi-automatic} if and only if it is $p$-automatic (resp. $p$-quasi-automatic) in the sense of 
\cite[Definition~4.1.2]{k-series-automata}.
\end{lemma}
\begin{proof}
It suffices to observe that when $K = \overline{\FF}_p$, the linear-algebraic data defining any $p$-automatic function can all be realized over a finite subfield of $K$.
\end{proof}

\begin{remark} \label{R:automatic properties}
As in \cite{k-series-automata}, we may verify that $K (( t^{\QQ} ))_{\aut}$
is stable under a number of basic operations.
\begin{enumerate}
\item[(a)]
The set $K (( t^{\QQ} ))_{\aut}$ is a $K$-vector subspace of $K (( t^{\QQ} ))$ containing $K$. By Lemma~\ref{L:any ab}, it is also a $K[t]$-submodule.
\item[(b)]
By Remark~\ref{R:tensor product},
the termwise product of two $p$-automatic functions is again $p$-automatic.
In particular, $K (( t^{\QQ} ))_{\aut}$ is stable under \emph{Hadamard products}: if it contains $\sum_i x_i t^i, \sum_i y_i t^i$, then it also contains
$\sum_i x_i y_i t^i$.
\item[(c)]
By (b), $K (( t^{\QQ} ))_{\aut}$ is stable under any substitution of the form $t \mapsto \mu t$ for $\mu \in K^\times$. (Note that this substitution is not specified by $\mu$ alone, as one must also choose a coherent sequence of roots of $\mu$.)
\item[(d)]
By Lemma~\ref{L:any ab}, $K (( t^{\QQ} ))_{\aut}$ is stable under
any substitution of the form $t \mapsto t^i$ for $i \in \QQ_{>0}$. 
Since it is also stable under automorphisms of $K$ applied coefficientwise,
it is thus stable under $\varphi$ and $\varphi^{-1}$.
\item[(e)]
For $x = \sum_i x_i t^i \in K((t^{\QQ}))$ and $r \in \RR$, define the \emph{$r$-truncation} of $x$ to be the series $\sum_{i<r} x_i t^i$. With this definition, if $x$ is $p$-quasi-automatic, then so is its $r$-truncation: Lemma~\ref{L:accumulation point} implies that 
$\sup\{i<r: x_i \neq 0\} \in \QQ$ and so we may reduce to the case $r\in \QQ$,
which follows from (b) plus Example~\ref{exa:Dedekind cut}.
\item[(f)]
Any finite linear combination $\sum_i y_i z_i$ in which $y_i, z_i$ are $p$-automatic, $y_i \in K((t))$, and $z_i$ is supported on $[0,1)$ is $p$-automatic. Conversely, any $p$-automatic series can be written as such a linear combination by choosing a basis of the vector space $V$ used in the construction.
\end{enumerate}
One result which is nontrivial to check is that $K (( t^{\QQ} ))_{\aut}$ 
is closed under multiplication; for $K = \overline{\FF}_p$ this is shown in \cite[Lemma~7.2.2]{k-series-automata} using the fact that addition of reversed base-$p$ expansions can be described using a uniform transducer in the style of Example~\ref{exa:p-adic arithmetic}. The general case can be proved by a similar but even more complicated argument, which we omit here; the closure under multiplication will anyway follow
\emph{a posteriori} from Theorem~\ref{T:integral closure of rational field}, and even from an intermediate result (see Remark~\ref{R:subring}).

Another nontrivial assertion, which has no analogue in \cite{k-series-automata}, is that for any algebraically closed field $K'$ containing $K$, within $K'((t^{\QQ}))$ one has
$K((t^{\QQ})) \cap K'((t^{\QQ}))_{\aut} = K((t^{\QQ}))_{\aut}$; that is, automaticity is insensitive to extensions of the coefficient field. This will again follow \emph{a posteriori} from Theorem~\ref{T:integral closure of rational field}.
\end{remark}

\section{Algebraicity of automatic generalized power series}

We now follow the proof of \cite[Proposition~5.1.2]{k-series-automata} to deduce that $p$-quasi-automatic series are integral over $K(t)$.

\begin{lemma} \label{L:semilinear solution}
Let $F \subset F'$ be an inclusion of algebraically closed fields of characteristic $p$. Let $V$ be a finite-dimensional $F$-vector space, put $V' = V \otimes_F F'$,
and define inclusions $\End_{\varphi}(V) \to \End_{\varphi}(V')$, $\End_{\varphi^{-1}}(V) \to \End_{\varphi^{-1}}(V')$ by formal extension of semilinear maps.
If $\bv \in V'$ satisfies $\bv = f(\bv)$ for some $f \in \End_{\varphi}(V) \cup \End_{\varphi^{-1}}(V)$, then $\bv \in V$.
\end{lemma}
\begin{proof}
See \cite[Lemma~3.3.5]{k-series-automata}.
\end{proof}

\begin{lemma} \label{L:automatic to algebraic}
If $x \in K (( t^\QQ ))_{\aut}$, then $x$ is integral over $K(t)$.
\end{lemma}
\begin{proof}
We may assume that $x$ is $p$-automatic.
Let $L$ be the integral closure of $K(t)$ in $K((t^{\QQ}))_{\aut}$.
Choose data as in Definition~\ref{D:biautomatic} with $V = K$.
By Lemma~\ref{L:flatten composed}, we may assume without loss of generality that $f_1$ is the $\varphi$-composed function associated to some map $\tau_1: \Sigma_p \to \End_{\varphi}(V')$ and that $f_2$ is the $\varphi^{-1}$-composed function associated to some map $\tau_2: \Sigma_p \to \End_{\varphi^{-1}}(V')$.
We may formally extend $\iota, \pi$ to $K((t^{\QQ}))$-linear maps
\[
\tilde{\iota}: K((t^{\QQ})) \to V' \otimes_K K((t^{\QQ})), \qquad
\tilde{\pi}: V' \otimes_K K((t^{\QQ})) \to K((t^{\QQ})).
\]
Define $g_1, g_2 \in \End_{\ZZ}(V' \otimes_K K((t^{\QQ})))$ by the formulas
\[
g_1 = \sum_{s \in L_p^0} f_1(\rev(s)), \qquad
g_2 = \sum_{s \in L_p^0} f_2(s).
\]
We then have
\[
g_1 = 1 + \sum_{s' \in L_p^0} \sum_{a=0}^{p-1}
f_1(\rev(as')) = 1 + \sum_{a=0}^{p-1} g_1 \circ \tau_1(a).
\]
By Lemma~\ref{L:semilinear solution},
this equality forces $g_1 \in \End_{\ZZ}(V' \otimes_K L)$.
By similar considerations, we have $g_2 \in \End_{\ZZ}(V' \otimes_K L)$.
By writing
\[
x = \tilde{\pi} \circ g_1 \circ g_2 \circ \tilde{\iota},
\]
we see that $x \in L$, as desired.
\end{proof}

\begin{remark} \label{R:use algebraic to equate}
Lemma~\ref{L:automatic to algebraic}
can be used to give an alternate proof of the ``only if'' implication of Lemma~\ref{L:compatibility}.
Namely, if $x$ is $p$-automatic in the sense
of Definition~\ref{D:p-quasi-automatic}, then it is algebraic over $K(t)$ by Lemma~\ref{L:automatic to algebraic}.
In particular, $x$ is algebraic over $\FF_p(t)$, so by
\cite[Proposition~7.3.4]{k-series-automata} it is $p$-quasi-automatic in the sense of
\cite[Definition~4.1.2]{k-series-automata}. Since $x$ is supported on $\ZZ[p^{-1}]_{\geq 0}$, we may apply \cite[Lemma~2.3.6]{k-series-automata} to see that it is also $p$-automatic in the sense of \cite[Definition~4.1.2]{k-series-automata}.
\end{remark}

\section{Automaticity of algebraic power series via decimation}

Our next step is to derive a generalization of Christol's theorem to an arbitrary algebraically closed field $K$, using the method of \emph{decimation}. This construction turns out to have strong geometric meaning; see Remark~\ref{R:Cartier operator}.

\begin{defn}
Let $L_0$ be either $K(t)$ or $K((t))$, so that $\varphi(L_0)$ is respectivetly $K(t^p)$ or $K((t^p))$. Let $L$ be a finite separable extension of $L_0$.
Then the finite separable extension $\varphi(L)$ of $\varphi(L_0)$ and the purely inseparable extension $L_0$ of $\varphi(L_0)$ must be linearly disjoint.
Consequently, the basis $1,t,\dots,t^{p-1}$ of $L_0$ over $\varphi(L_0)$ is also a basis of $L$ over $\varphi(L)$.
We thus have functions $s_0,\dots,s_{p-1}: L \to L$ uniquely determined by the formula
\[
x = \sum_{i=0}^{p-1} s_{i}(x)^p t^i \qquad (x \in L).
\]
Since a derivation in characteristic $p$ kills any $p$-th power, we also have
\begin{gather} 
\label{eq:decimate2}
\frac{d^i x}{dt^i} = \sum_{j=i}^{p-1} j(j-1)\cdots(j-i+1) s_j(x)^p t^{j-i} \qquad (i=0,\dots,p-1) \\
\label{eq:decimate1}
\frac{d^{p-1} (t^{p-1-i} x)}{dt^{p-1}} = -s_i(x)^p \qquad (i=0,\dots,p-1).
\end{gather}
\end{defn}

\begin{theorem} \label{T:gen Christol}
An element $x \in K \llbracket t \rrbracket$ is integral over $K(t)$ if and only if it is $p$-automatic.
\end{theorem}
\begin{proof}
By Lemma~\ref{L:automatic to algebraic} we need only check the ``only if'' direction. 
Suppose $x \in K \llbracket t \rrbracket$ is algebraic over $K(t)$.
To check that $x$ is $p$-automatic, by Remark~\ref{R:automatic properties}(e)
we may instead check that $\varphi^n(x)$ is $p$-automatic for some conveniently large integer $n$; we may thus assume that $x$ belongs to a finite separable extension $L$ of $K(t)$.
Let $C$ be a smooth, projective, irreducible curve over $K$ with function field $L$. 
Let $V$ be the minimal $K$-subspace of $L$ containing $x$ and closed under $s_0,\dots,s_{p-1}$. Choose a closed point $P \in C$, an element $y \in L$, and a value $i\in \{0,\dots,p-1\}$.
If $\ord_P(x) \geq 0$ and $\ord_P(dt) = 0$, by applying \eqref{eq:decimate2} with $i=p-1,\dots,0$ in turn we obtain
\[
\ord_P(y) \geq 0 \Longrightarrow \ord_P(s_i(y)) \geq 0.
\]
This excludes only finitely many $P$, for which we apply \eqref{eq:decimate1} to deduce that
\begin{equation} \label{eq:decimate pole}
\ord_P(s_{i}(y)) \geq \left\lceil \frac{\ord_P(y)-(p-1)+(p-1-i) \ord_P(t) - (p-1)\ord_P(dt)}{p} \right\rceil.
\end{equation}
It follows that for each $P$ there exists some integer $C_P$ such that
$C_P = 0$ for all but finitely many $P$ and
\[
\ord_P(y) \geq C_P \qquad \mbox{for all $y \in V$}.
\]
It follows that $V$ is a finite-dimensional $K$-vector space, from which we deduce easily that $x$ is $p$-automatic.
\end{proof}
\begin{cor} \label{C:module for series}
The set $K((t^{\QQ}))_{\aut}$ is a module over $K((t^{\QQ}))_{\aut} \cap K((t))$.
\end{cor}
\begin{proof}
By Remark~\ref{R:automatic properties}(a), $K((t^{\QQ}))_{\aut}$ is closed under addition.
By Theorem~\ref{T:gen Christol}, $K((t^{\QQ}))_{\aut} \cap K((t))$ is a ring (namely the integral closure of $K(t)$ in $K((t))$). It thus remains to verify that 
for $x \in K((t^{\QQ}))_{\aut}$, $y \in K((t^{\QQ}))_{\aut} \cap K((t))$,
we have $xy \in K((t^{\QQ}))_{\aut}$. To check this, we may apply 
Lemma~\ref{L:any ab} to reduce to the case where $x,y$
are both $p$-automatic. As in Remark~\ref{R:automatic properties}(f),
we may express $x$ as a finite linear combination $\sum_i y_i z_i$ in which
$y_i \in K((t))$ is $p$-automatic, and $z_i$ is $p$-automatic and supported on $[0,1)$;
by the same remark, $xy = \sum_i (yy_i) z_i$ is $p$-automatic.
\end{proof}

We append a remark suggested by David Speyer in \cite{speyer}.
\begin{remark} \label{R:Cartier operator}
Note that $\Omega_L$ (the module of absolute K\"ahler differentials) is a one-dimensional vector space generated by $dt$. It was shown by Tate \cite{tate} that the operator
$\mathcal{C}: \Omega_{L/K} \to \Omega_{L/K}$ taking $x\,dt$ to $s_{p-1}(x)\,dt$ is independent of the choice of the local coordinate $t$. This map is called the \emph{Cartier operator} and is of great importance in characteristic-$p$ algebraic geometry;
it can be even used to control the complexity (i.e., the number of states)
of the automata arising in Christol's theorem,
as shown by Bridy \cite{bridy}. However, we will have no further use for it here.
\end{remark}

\begin{remark} \label{R:recover Christol}
Let $\FF$ be a finite field. By Theorem~\ref{T:gen Christol},
an element $x = \sum_{i=0}^\infty x_i t^i \in \FF \llbracket t \rrbracket$ is integral over $\FF(t)$ if and only if it is $p$-quasi-automatic.
By Lemma~\ref{L:compatibility}, $x$ is $p$-quasi-automatic if and only if the function
$L_p \to \FF$ taking $s$ to $x_{\|s\|}$ is $p$-automatic.
That is, Theorem~\ref{T:gen Christol} does indeed generalize Christol's theorem
(Theorem~\ref{T:christol}).
\end{remark}

\section{Automaticity of truncated algebraic series}

We now extend the ``algebraic implies automatic'' implication of Theorem~\ref{T:gen Christol} from power series to truncations of generalized power series. As in the incorrect argument given in \cite{k-series1} (see Remark~\ref{R:not TR}), the key calculation here involves Artin-Schreier extensions; while this calculation can be made directly, we prefer to use a trick to reduce to the case of ordinary power series.
\begin{lemma} \label{L:automatic AS}
If $x \in K((t^{\QQ}))_{\aut}, y \in K((t^{\QQ}))$ satisfy $y^p -y =x$, then any truncation of $y$ is $p$-quasi-automatic.
\end{lemma}
\begin{proof}
By Lemma~\ref{L:any ab} and Remark~\ref{R:automatic properties}(e), we may assume that the support of $x$ is contained in either $(-\infty,0)$ or $(0, \infty)$. In the latter case, we have
\[
y = c - x - x^{p} - \cdots
\]
for some $c \in \FF_p$; since only finitely many summands contribute to any truncation, the claim is clear.

So let us suppose hereafter that $x$ is supported on $(-\infty, 0)$. By Lemma~\ref{L:any ab}, we may further reduce to the case that $tx$ is $p$-automatic and $x$ is supported on $(-1,0)$. Then
\[
y = x^{1/p} + x^{1/p^2} + \dots + c
\]
for some $c \in \FF_p$; there is no harm in assuming that $c=0$. Put
$\tilde{x} = t x$
and $\tilde{y} = t^{1/p} y^{1/p}$, so that $\tilde{x}$ is supported on $(0,1) \cap \ZZ[p^{-1}]$, $\tilde{y}$ is supported on $(0, 1/p) \cap \ZZ[p^{-1}]$, and 
\begin{equation} \label{eq:as1}
\tilde{y}^p - t^{(p-1)/p} \tilde{y}  = \tilde{x}.
\end{equation}
Write $\tilde{x} = \sum_i \tilde{x}_i t^i$, $\tilde{y} = \sum_i \tilde{y}_i t^i$,
and define $X, Y \in K \llbracket t \rrbracket$ by the formulas
\[
X_{\left| s \right|} = \tilde{x}_{\left\| .\rev(s) \right\|},
\qquad
Y_{\left| s \right|} = \tilde{y}_{\left\| .\rev(s) \right\|}
\qquad (s \in L_p^0).
\]
By Lemma~\ref{L:any ab}
 and Corollary~\ref{C:automatic reverse}, $X$ is $p$-automatic.
Now note that in \eqref{eq:as1}, the computation of the left side involves no carries in base $p$-arithmetic; we may thus formally reverse to deduce the new identity
\[
Y^{1/p} - t^{p-1} Y = X.
\]
By Theorem~\ref{T:gen Christol}, $Y$ is $p$-automatic; by Lemma~\ref{L:any ab} and Corollary~\ref{C:automatic reverse} again, $\tilde{y}$ is $p$-automatic.
\end{proof}
\begin{cor} \label{C:automatic AS}
Let $P(T) \in K[T]$ be a nonzero polynomial.
If $x,y \in K (( t^{\QQ} ))$ are such that $x$ is $p$-quasi-automatic and $P(\varphi)(y) = x$, then every truncation of $y$ is $p$-quasi-automatic.
\end{cor}
\begin{proof}
In case $P(T) = T - \mu$ for some $\mu \in K$, this follows from Lemma~\ref{L:automatic AS} upon replacing $x$ with a sufficiently close approximation (using the fact that the roots of a polynomial vary continously with the coefficients, as in \cite[Lemma~8]{k-series2}).
The general case reduces to this case via \cite[Corollary~6.4.6]{k-series-automata}.
\end{proof}

The following result is analogous to \cite[Proposition~7.3.3]{k-series-automata}, although the proof is more in the spirit of the transfinite Newton algorithm described in
\cite[Proposition~1]{k-series2}.
\begin{lemma} \label{L:truncated algebraic to automatic}
Suppose that $x = \sum_i x_i t^i \in K((t^{\QQ}))$ is supported on $\ZZ[p^{-1}]_{\geq 0}$ and is integral over $K((t))$. Then every truncation of $x$ is $p$-automatic.
\end{lemma}
\begin{proof}
For $r \in \RR$, let $x_{(r)}$ denote the $r$-truncation of $x$.
Let $S$ be the set of $r \in \RR$ for which $x_{(r)}$ is $p$-automatic.
Since $x$ has well-ordered support, $S$ is nonempty and does not contain a largest element. It thus suffices to deduce a contradiction assuming that $S$ admits a finite supremum $r$ (not necessarily in $\QQ$), which necessarily does not belong to $S$.

By Ore's lemma (see \cite[Lemma~12.2.3]{allouche-shallit} or \cite[Lemma~3.3.4]{k-series-automata}), there exists a monic polynomial $P(T) = \sum_{j=0}^d P_j T^j$ with coefficients in $K((t))$ such that
$P(\varphi)(x) = 0$. Let $v_t$ denote the $t$-adic valuation on $K((t))$ and put $c = \min_j \{v_t(P_j) +rp^j\}$. 
For $j=0,\dots,d$, let $Q_j$ be the coefficient of $t^{c-rp^j}$ in $P_j$
(interpreted as $0$ if $c - rp^j \notin \ZZ$)
and put
$Q(T) = \sum_{j=0}^d Q_j T^j$; by the definition of $c$, $Q(T) \neq 0$.

By writing the $c$-truncation of $P_j x^{p^j}$ as 
$\sum_{m \geq c-rp^j} P_{j,m} t^m (x_{((c-m)/p^j)})^{p^j}$,
we see that it has the form $y_j + z_j$ where
$y_j = Q_j t^{c-rp^j} (x_{(r)})^{p^j}$ (again interpreted as $0$ if $c-rp^j \notin \ZZ$) and $z_j$ is $p$-automatic. Put $y = \sum_{j=0}^d y_j$; since $P(\varphi)(x) = 0$, $y$ is $p$-automatic.
If $Q$ consists of a single monomial, then $y = Q_j t^{c-rp^j} (x_{(r)})^{p^j}$ for some $j$ and so $x_{(r)}$ is $p$-quasi-automatic; otherwise, $r,c \in \QQ$ 
and $t^{-c} y = Q(\varphi)(t^{-r} x_{(r)})$, so
Corollary~\ref{C:automatic AS} implies that $x_{(r)}$ is $p$-quasi-automatic. In either case, we have $r \in S$, a contradiction.
\end{proof}
\begin{cor} \label{C:truncated algebraic to automatic}
If $x \in K((t^{\QQ}))$ is integral over $K((t))$ and has bounded support, then $x$ is $p$-quasi-automatic.
\end{cor}
\begin{proof}
By Lemma~\ref{L:any ab}, we may reduce to the case where $x$ is supported on $\ZZ[p^{-1}]_{\geq 0}$. We may then deduce the claim from
Lemma~\ref{L:truncated algebraic to automatic}.
\end{proof}

\begin{remark} \label{R:subring}
At this point, one can use Corollary~\ref{C:module for series} and Corollary~\ref{C:truncated algebraic to automatic} to deduce that $K((t^{\QQ}))_{\aut}$ is a subring of $K((t^{\QQ}))$. However, we will not use this explicitly.
\end{remark}

\section{The main theorems}

At this point, it is merely a matter of tying up loose ends to establish our main results.
We start with the characterization of the integral closure of $K(t)$ in $K((t^{\QQ}))$.

\begin{lemma} \label{L:Krasner}
Suppose that $x = \sum_i x_i t^i \in K((t^{\QQ}))$ is integral over $K((t))$. Then for $r \in \RR$ sufficiently large, the $r$-truncation $y_{(r)}$ of $x$ is $p$-quasi-automatic,
is integral over $K(t)$, 
 and generates a finite extension of $K((t))$ containing $x$.
\end{lemma}
\begin{proof}
By Remark~\ref{R:automatic properties}(d), we may assume that $x$ generates a finite separable extension of $K((t))$.
For all $r \in \RR$, $y_{(r)}$ is $p$-quasi-automatic by Lemma~\ref{L:truncated algebraic to automatic}
and hence integral over $K(t)$ by Lemma~\ref{L:automatic to algebraic}.
For sufficiently large $r$, we may apply Krasner's lemma to the minimal polynomial of $x$ to see that $x$ and $y_{(r)}$ generate the same extension of $K((t))$.
\end{proof}

\begin{lemma} \label{L:base extend integrally closed}
Let $F \subseteq F'$ be an inclusion of fields such that $F$ is integrally closed in $F'$.
Let $E$ be a finite extension of $F$. 
\begin{enumerate}
\item[(a)] The ring $E' = E \otimes_F F'$ is again a field.
\item[(b)]
The field $E$ is integrally closed in $E'$.
\end{enumerate}
\end{lemma}
\begin{proof}
We first check (a) assuming that $E = F(y)$ for some $y \in E$; note that this entirely covers the case where $E/F$ is separable (by the primitive element theorem).
Let $P \in F[T]$ be the minimal polynomial of $y$ over $F$.
In any algebraically closed field containing $F$, the roots of $P$ are all integral over $F$; consequently, for any factorization of $P$ into monic polynomials in $F'[T]$, the coefficients of the factors are also integral over $F$. Since $F$ is integrally closed in $F'$, it follows that $P$ remains irreducible in $F'[T]$; hence $E' = E \otimes_F F' = F'[T]/(P)$ is again a field.

To prove (b), suppose that $x \in E'$ is integral over $E$. 
Let $Q \in E[T]$ be the minimal polynomial of $x$ over $E$.
Suppose for the moment that we can find a finite extension $E_1$ of $F$ for which the conclusion of (a) holds, that is, for which
$E'_1 = E_1 \otimes_F F'$ is again a field.
Since $Q$ already splits into linear factors over $E_1$ and $x$ is a root of $Q$ in the field $E'_1$, $x$ must belong to $E_1$. But $x$ also belongs to $E' = E \otimes_F F'$, and inside $E_1 \otimes_F F'$ we have $E_1 \cap (E \otimes_F F') = E$.
Hence $E$ is integrally closed in $E'$, as desired.

Note that the preceding argument again applies in case $x$ generates a purely inseparable extension of $E$, as then we may take $E_1$ to be $E(x)$. Consequently, both (a) and (b) hold
if $E/F$ is purely inseparable, and so (a) holds if $E/F$ is a separable extension of a purely inseparable extension. However, any polynomial can be split by such an extension, so both (a) and (b) hold in general.
\end{proof}

\begin{cor} \label{C:base extend integrally closed}
Let $F \subseteq F' \subseteq F''$  be inclusions of fields. 
Suppose that $x,y \in F''$
are integral over $F$ and that $x \in F'(y)$. Put $m = [F'(y):F']$ and write
$x = \sum_{i=0}^{m-1} a_i y^i$ with $a_i \in F'$. Then the $a_i$ are themselves integral over $F$.
\end{cor}
\begin{proof}
We may assume that $F$ is integrally closed in $F'$.
By Lemma~\ref{L:base extend integrally closed},
$F'(y) = F(y) \otimes_F F'$ and $F(y)$ is integrally closed in $F'(y)$. Since $x \in F'(y)$
is integral over $F$, it follows that $x \in F(y)$ and $a_i \in F$, as claimed.
\end{proof}

\begin{theorem} \label{T:integral closure of rational field}
The integral closure of $K(t)$ in $K((t^{\QQ}))$ equals $K((t^{\QQ}))_{\aut}$.
\end{theorem}
\begin{proof}
By Lemma~\ref{L:automatic to algebraic}, $K((t^{\QQ}))_{\aut}$ is contained in the integral closure, so it remains to prove the converse. 
We imitate the proof of \cite[Proposition~7.3.4]{k-series-automata}.
Choose $x \in K((t^{\QQ}))$ integral over $K(t)$. 
By Lemma~\ref{L:Krasner}, there exists a truncation $y$ of $x$ that is $p$-quasi-automatic and integral over $K(t)$, and moreover has the property that $x$ and $y$ generate the same finite extension of $K((t))$.

Note that any power of $y$ has bounded support, and hence is $p$-quasi-automatic by 
Corollary~\ref{C:truncated algebraic to automatic}.
Since $x \in K((t))(y)$, we may apply Corollary~\ref{C:base extend integrally closed}
to write $x$ as $\sum_{i=0}^{m-1} a_i y^i$ for some nonnegative integer $m$ and some $a_i$ in the integral closure of $K(t)$ in $K((t))$. By Theorem~\ref{T:gen Christol}, each $a_i$ is $p$-quasi-automatic;
by  Corollary~\ref{C:module for series}, $x$ is $p$-quasi-automatic.
\end{proof}

As a corollary, we may immediately describe the completion of this integral closure.
\begin{theorem} \label{T:main completed}
The completed integral closure of $K(t)$ (or equivalently $K((t))$) in $K((t^{\QQ}))$ consists of those $x = \sum_i x_i t^i$ all of whose truncations are $p$-quasi-automatic.
\end{theorem}
\begin{proof}
The completed integral closure contains the set in question by Lemma~\ref{L:automatic to algebraic} and is contained in this set by Theorem~\ref{T:integral closure of rational field}.
\end{proof}

We next describe the integral closure of $K((t))$ in $K((t^{\QQ}))$.
\begin{theorem} \label{T:main}
The integral closure of $K((t))$ in $K((t^{\QQ}))$ consists of those 
$x = \sum_i x_i t^i$ such that the series 
\[
y_n = \sum_{i \in [0,1) \cap \QQ} x_{n+i} t^i \qquad (n \in \ZZ)
\]
are all $p$-quasi-automatic and generate a finite-dimensional $K$-vector space.
\end{theorem}
\begin{proof}
If the $y_n$ are all $p$-quasi-automatic, then each one is integral over $K(t)$ by
Lemma~\ref{L:automatic to algebraic}. If they moreover generate a finite-dimensional $K$-vector space, then $x$ can be written as a $K((t))$-linear combination of the $y_i$, and thus is integral over $K((t))$. Conversely, if $x$ is integral over $K((t))$,
then by Lemma~\ref{L:truncated algebraic to automatic}, each $y_n$ is $p$-quasi-automatic.
By Lemma~\ref{L:Krasner}, for $N$ sufficiently large, $x$ belongs to the finite extension of $K((t))$ generated by $\{y_n: n <N\}$;
then the $y_n$ for all $n \geq N$ belong to the $K$-vector space generated by $\{y_n: n < N\}$. This proves the claim.
\end{proof}

\begin{remark}
Another way to phrase the condition in Theorem~\ref{T:main}  is that the $y_n$ can be generated as in Definition~\ref{D:p-quasi-automatic} in such a way that everything in Definition~\ref{D:biautomatic} is chosen uniformly as $n$ varies except for the maps
 $\iota, \pi$.
\end{remark}

\begin{remark}
We note that Theorem~\ref{T:main} immediately implies Proposition~\ref{P:TR algebraic},
thus reconfirming the implication ``twist-recurrent implies algebraic'' of \cite[Theorem~8]{k-series1}.
\end{remark}

\section{Corollaries}
\label{sec:corollaries}

We next recover a number of corollaries of our main theorems.
Many (but not all) of these were stated as results in \cite{k-series1},
but in light of the failure of \cite[Theorem~8]{k-series1} we give proofs independent of \cite{k-series1}.

The following result is stated as \cite[Theorem~6]{k-series1}, where it is also observed that the subring described is not itself a field. (In some sense, this is the closest approximation of the integral closure that can be defined solely in terms of supports.)
\begin{cor} \label{C:sabc}
The subring of $K((t^{\QQ}))$ consisting of series supported within some $S_{a,b,c}$
(as defined in Definition~\ref{D:sabc})
 contains the integral closure of $K((t))$ in $K((t^{\QQ}))$. 
\end{cor}
\begin{proof}
This follows from Theorem~\ref{T:main} because any $p$-quasi-automatic series is supported in some $S_{a,b,c}$.
\end{proof}

The following statement includes \cite[Corollary~9]{k-series1}.
\begin{cor} \label{C:perfect coefficient subfield}
Let $F$ be a perfect subfield of $K$. Then the integral closure of $F(t)$ (resp.\ $F((t))$) in $K((t^{\QQ}))$ consists of those elements of the integral closure in $K((t^{\QQ}))$ with coefficients contained in a finite extension of $F$.
\end{cor}
\begin{proof}
In Theorem~\ref{T:main}, each element of the integral closure in $K((t^{\QQ}))$ with bounded support is specified by some discrete data and a finite number of initial coefficients, the others being determined by linear-algebraic constraints. Consequently, if these coefficients belong to $F$, so do all of the others.
\end{proof}

For an imperfect subfield of $K$, we have the following incomplete classification.
This statement is \cite[Corollary~10]{k-series1} except that therein, ``bounded below'' is incorrectly substituted for ``bounded above''.
\begin{cor} \label{C:imperfect coefficient subfield}
Let $F$ be a not necessarily perfect subfield of $K$. If $x \in K((t^{\QQ}))$ is algebraic over $F((t))$, then the following conditions must hold.
\begin{enumerate}
\item[(a)]
There exists a finite extension $F'$ of $F$ whose perfect closure contains all of the $x_i$.
\item[(b)]
For each $i$, let $f_i$ be the smallest nonnegative integer such that $\varphi^{f_i}(x_i) \in F'$. Then $f_i - v_p(i)$ is bounded above.
\end{enumerate}
\end{cor}
\begin{proof}
This follows from Theorem~\ref{T:main} as in the proof of Corollary~\ref{C:perfect coefficient subfield}.
\end{proof}

The following includes \cite[Corollary~12]{k-series1}.
\begin{cor} \label{C:full truncation}
Suppose that $x \in K((t^{\QQ}))$ is integral over $K((t))$. Then every truncation of $x$ is integral over $K(t)$, and hence also over $K((t))$.
\end{cor}
\begin{proof}
This is immediate from Theorem~\ref{T:integral closure of rational field} and Theorem~\ref{T:main}.
\end{proof}

The following statement replaces \cite[Corollary~13]{k-series1}.
\begin{defn}
Suppose that $K$ is complete with respect to a multiplicative nonarchimedean norm $\left| \bullet \right|_K$.
We say that $x = \sum_i x_i t^i \in K((t^{\QQ}))$ has \emph{positive radius of convergence} if there exists $r>0$ such that $r^i \left|x_i\right|_K \to 0$ as $i \to \infty$.
\end{defn}

\begin{cor} \label{C:radius of convergence}
Suppose that $K$ is complete with respect to a multiplicative nonarchimedean norm $\left| \bullet \right|_K$.
Let $K((t))^c$ be the subfield of $K((t))$ consisting of series with positive radius of convergence. Then the integral closure of $K((t))^c$ in $K((t^\QQ))$ consists of those $x$
which are integral over $K((t))$ (and thus may be described as in Theorem~\ref{T:main})
and have positive radius of convergence.
\end{cor}
\begin{proof}
One checks easily using Hensel's lemma that if $x$ is integral over $K((t))^c$,
then $x$ has positive radius of convergence. This proves the claim.
\end{proof}
The following corollary is new. Note that a special case was treated in Lemma~\ref{L:automatic AS} by directly identifying an algebraic relation satisfied by the transformed series; by contrast, in the general case it seems quite tricky to make such an identification. Note also that the condition on the support of $x$ is quite restrictive; compare Remark~\ref{R:bad support}.
\begin{cor}
For $i \in \ZZ[p^{-1}]_{\geq 0}$, let $\rev(i)$ be the function defined so that
$\rev(\|s\|) = \|\rev(s)\|$ for $s \in L_p$. Suppose that $x = \sum_i x_i t^i \in K((t^{\QQ}))$ is supported on a subset $S$ of $\ZZ[p^{-1}]_{\geq 0}$ such that $\rev(S)$ is also well-ordered. Then $x$ is algebraic over $K(t)$ if and only if $\sum_i x_{\rev(i)} t^i$ is.
\end{cor}
\begin{proof}
This is immediate from Theorem~\ref{T:integral closure of rational field}.
\end{proof}

The following is a theorem of Vaidya \cite[Lemma~4.1.1]{vaidya} generalizing an older result of Huang and Stef\u{a}nescu \cite[Theorem~2]{k-series1}.
\begin{cor}
The series $x = \sum_{i=1}^\infty x_i t^{p^{-i}} \in K((t^{\QQ}))$ is integral over $K((t))$ if and only if the sequence $\{x_i\}$ satisfies an LRR, in which case $x$ is also integral over $K(t)$.
\end{cor}
\begin{proof}
Immediate from Theorem~\ref{T:integral closure of rational field} and Theorem~\ref{T:main}.
\end{proof}

The following corollary generalizes a result of Deligne \cite{deligne} (extending a previous result of Furstenberg \cite{furstenberg}): the integral closure of $K(t)$ in $K((t))$ is stable under Hadamard products. (A more elementary proof has been given by
Sharif and Woodcock \cite{sharif-woodcock}.)
\begin{cor} \label{C:integral closure Hadamard}
The integral closure of $K(t)$ (resp.\ $K((t))$) in $K((t^{\QQ}))$ is stable under Hadamard products.
\end{cor}
\begin{proof}
Immediate from Remark~\ref{R:automatic properties}
and Theorem~\ref{T:integral closure of rational field}
(resp.\ Theorem~\ref{T:main}).
\end{proof}

We now specialize to the case where $K = \overline{\FF}_p$. In this context, we have already recovered Christol's theorem (see Remark~\ref{R:recover Christol}).
We next recover Theorem~\ref{T:automata1} in a similar fashion.
\begin{remark} \label{R:recover automata1}
Let $F$ be a finite field. By Theorem~\ref{T:integral closure of rational field},
an element $x = \sum_i x_i t^i \in F((t^{\QQ}))$ is integral over $F(t)$ if and only 
it is $p$-quasi-automatic. By Lemma~\ref{L:compatibility}, this is equivalent to the conditions of Theorem~\ref{T:automata1}.
\end{remark}

We next recover \cite[Theorem~15]{k-series1}, and hence also \cite[Theorem~2]{k-series1}.
\begin{theorem} \label{T:finite field2}
A series $x = \sum_i x_i t^i \in \overline{\FF}_p((t^{\QQ}))$
is integral over $\overline{\FF}_p((t))$  if and only if the following conditions hold.
\begin{enumerate}
\item[(a)]
There exist $a,b,c$ such that $x$ is supported on $S_{a,b,c}$.
\item[(b)]
For some (hence any) $a,b,c$ as in (a),
there exist integers $M,N$ with the following property:
for each integer $m \geq -b$, the function $f_m: T_c \to K$ given by $f_m(z) = x_{(m+z)/a}$ has the property that every sequence as in \eqref{eq:TR sequences} becomes periodic of period $N$ after at most $M$ terms.
\end{enumerate}
\end{theorem}
\begin{proof}
This is immediate from Theorem~\ref{T:integral closure of rational field}.
\end{proof}

We next recover the statement of \cite[Theorem~8]{k-series1} in the case $K = \overline{\FF}_p$.
\begin{theorem} \label{T:finite field}
A series $x = \sum_i x_i t^i \in \overline{\FF}_p((t^{\QQ}))$
is integral over $\overline{\FF}_p((t))$  if and only if it is twist-recurrent in the sense of Definition~\ref{D:twist-recurrent}.
\end{theorem}
\begin{proof}
This is immediate from Theorem~\ref{T:finite field}.
\end{proof}

\section{Zero sets} \label{sec:zero sets}

The Skolem-Mahler-Lech theorem asserts that the zero set of a linear recurrent sequence over a field of characteristic zero consists of a finite union of arithmetic progressions plus a finite set. Over a field of characteristic $p$, such a zero set can be more complicated; for example, the sequence $\{(1+t)^n - 1 - t^n\}_{n=0}^\infty$ over $\FF_p(t)$ has zero set $\{1,p,p^2,\dots\}$. The zero sets of linear recurrent sequences in characteristic $p$ have been classified by Derksen \cite[Theorem~1.8]{derksen}; an intriguing feature of this classification is that every zero  set is the image under $\left| \bullet \right|$ of a regular language in $L_p^0$.

It has been observed by Adamczewski and Bell \cite{adamczewski-bell2} that such a statement also holds for the coefficients of algebraic power series. To prepare for the discussion of the corresponding statement for generalized power series, it is useful to formulate a stronger multivariate statement from \cite{adamczewski-bell2}, then give an alternate proof of the latter.

\begin{defn}
For $k$ a positive integer, let $L_p^{0,k} \subset (\Sigma_p^k)^*$ be the set of strings not beginning with $(0,\dots,0)$; this is evidently a regular language. For $s = (s_0,\dots,s_n) \in L_p^{0,k}$, define
\[
\left| s \right| = \left(\sum_{j=0}^n s_{j,h} p^{n-j} \right)_{h=1}^k;
\]
this defines a bijection $\left| \bullet \right|: L_p^{0,k} \to \ZZ_{\geq 0}^k$.
\end{defn}

The following is \cite[Theorem~1.4]{adamczewski-bell2}, but with a different proof.
\begin{theorem} \label{T:derksen}
For any $x = \sum_{n_1,\dots,n_k=0}^\infty x_{n_1,\dots,n_k} t_1^{n_1} \cdots t_k^{n_k} \in K \llbracket t_1,\dots,t_k \rrbracket$ which is integral over $K(t_1,\dots,t_k)$,
the set $\{(n_1,\dots,n_k) \in \ZZ_{\geq 0}: x_{n_1,\dots,n_k} = 0\}$ is the image under $\left| \bullet \right|$ of a regular language in $L_p^{0,k}$.
\end{theorem}
\begin{proof}
Note that the $x_{n_1,\dots,n_k}$ all belong to some finitely generated subfield $K_0$ of $K$; by Noether normalization, $K_0$ is finite over some subfield $K_1$ of the form $\FF_p(z_1,\dots,z_m)$. Let $L$ be the normal closure of $K_0$ over $K_1$. The set of elements of $K \llbracket t_1,\dots,t_k \rrbracket$ integral over $K(t_1,\dots,t_k)$ is closed under Hadamard products
\cite{sharif-woodcock}, so by replacing each $x_{n_1,\dots,n_k}$ with its norm from $L$ to $K_1$ we may reduce to the case $K_0 = K_1$. By rescaling $t$, we may further reduce to the case where the $x_{n_1,\dots,n_k}$ all belong to $\FF_p[z_1,\dots,z_m]$.
We may then view $x$ as a series
\[
\sum_{n_1,\dots,n_k,n'_1,\dots,n'_m=0}^\infty
x_{n_1,\dots,n_k,n'_1,\dots,n'_m} t_1^{n_1} \cdots t_k^{n_k} z_1^{n'_1} \cdots z_m^{n'_m}
\in
\FF_p \llbracket t_1,\dots,t_k, z_1,\dots,z_m \rrbracket
\]
which is integral over $\FF_p(t_1,\dots,t_k,z_1,\dots,z_m)$.
By Salon's multivariate analogue of Christol's theorem \cite{salon},
the latter is $p$-automatic; that is, the function $L_p^{0,k+m} \to \FF_p$ taking $s$ to $x_{|s|}$ is a finite-state function.
In particular, the set $\{s \in L_p^{0,k+m}: x_{|s|} = 0\}$ is a regular language;
we may conclude from this point by repeated application of Lemma~\ref{L:quantifier elimination}.
\end{proof}

Note that even when $k=1$, the proof of Theorem~\ref{T:derksen} involves series in more than one variable. Consequently, to make an analogous argument for generalized power series, we must introduce some notation about multivariate series; we give only a limited treatment here, as a complete ``algebraic equals automatic'' theorem in this context is somewhat difficult to formulate (see the discussion at the end of \cite{k-series-automata}).

\begin{defn}
For $k$ a positive integer, let $L_p^k \subseteq (\Sigma_p^k \cup \{.\})^*$ be the set of the strings of the form $s_1.s_2$ with $s_1 \in L_p^{0,k}$, $s_2 \in \rev(L_p^{0,k})$.
For $s = (s_0,\dots,s_n) \in L_p^k$, let $i$ be the unique index for which $s_i = .$, and define
\[
\left\| s \right\| = \left( \sum_{j=0}^{i-1} s_{j,h} p^{i-j-1} + \sum_{j=i+1}^n s_{j,h} p^{i-j} \right)_{h=1}^k
\]
\end{defn}

\begin{theorem} \label{T:zero set}
For $x = \sum_i x_i t^i \in K((t^{\QQ}))$ integral over $K(t)$ with support in $\ZZ[p^{-1}]_{\geq 0}$, the set 
$\{i \in \ZZ[p^{-1}]_{\geq 0}: x_i = 0\}$
is the image under $\left\| \bullet \right\|$ of a regular language.
\end{theorem}
\begin{proof}
By Corollary~\ref{C:perfect coefficient subfield},
the $x_i$ all belong to the perfect closure of some finitely generated subfield $K_0$ of $K$; by Noether normalization, $K_0$ is finite over some subfield $K_1$ of the form $\FF_p(z_1,\dots,z_m)$. Let $L$ be the normal closure of $K_0$ over $K_1$. 
By Corollary~\ref{C:integral closure Hadamard},
replacing each $x_i$ with its norm from the perfect closure of $L$ to the perfect closure of $K_1$ yields another element of $K((t^{\QQ}))$ integral over $K(t)$, so we may assume hereafter that $K_0 = K_1$. 

By Lemma~\ref{L:Krasner}, there exists a truncation $y$ of $x$ that is $p$-automatic
and integral over $K(t)$, and moreover has the property that $x$ and $y$ generate the same finite extension of $K((t))$.
Since $x \in K((t))(y)$, we may apply Corollary~\ref{C:base extend integrally closed}
to write $x$ as a polynomial in $y$ whose coefficients belong  to the integral closure of $K(t)$ in $K((t))$.

Let $S$ be the subring of $K((t))$ consisting of series whose coefficients belong to the perfect closure of $\FF_p[z_1,\dots,z_m]$.
Write $K((t))(y)$ as a tower of Artin-Schreier extensions of $K((t))$;
by rescaling $t$, we may further reduce to the case where $x \in S$ and these Artin-Schreier extensions lift to \'etale extensions of $S$.
Write $x$ as a formal sum
\[
\sum_{i,i'_1,\dots,i'_m \in \ZZ[p^{-1}]_{\geq 0}}
x_{i, i'_1,\dots,i'_m} t^i z_1^{i'_1} \cdots z_m^{i'_m}
\]
with coefficients in $\FF_p$; we may then argue as in the proof of 
Lemma~\ref{L:truncated algebraic to automatic} to see that
for any $r \in \RR$, the function
$L_p^{m+1} \to \FF_p$ taking $s$ to $x_{\left\|s \right\|}$ if the first component of $\left\|s \right\|$ is at most $r$
and 0 otherwise is a finite-state function.
Combining this result with Salon's theorem, we may deduce that
the function
$L_p^{m+1} \to \FF_p$ taking $s$ to $x_{\left\|s \right\|}$ is a finite-state function.
We again use Lemma~\ref{L:quantifier elimination} to finish.
\end{proof}

\section{An analogue in mixed characteristic}
\label{sec:mixed}

In \cite{k-series2}, an explicit description is given of the completed algebraic closure of a local field of mixed characteristics. This description depends on \cite[Theorem~8]{k-series1} via \cite[Theorem~10]{k-series2}; consequently, we may invoke Theorem~\ref{T:finite field} to see that \cite[Theorem~10, Theorem~11]{k-series2} remain correct when the local field has residue field contained in 
$\overline{\FF}_p$. (In particular, the remarks appearing after the latter result remain valid.) However, the results fail without this additional restriction.

We now use our results to give an alternate characterization of the completed algebraic closure in the language of finite automata.
\begin{defn}
For $n$ a positive integer or $\infty$, let $W_n$ denote the functor of $p$-typical Witt vectors of length $n$; we have a canonical identification 
\[
W_n(K((t^{\QQ}))) \cong W_n(K)((t^{\QQ}))^{\wedge},
\]
where the wedge indicates the $p$-adic completion if $n = \infty$ and otherwise has no effect. It is customary to omit the subscript $n$ when it equals $\infty$.

Using the Witt vector Frobenius $\varphi$ on $W_n(K)$, we may proceed by analogy with 
\S\ref{sec:composed functions} to define \emph{$\varphi$-automatic} functions $f: L_p^0 \to \End_{\ZZ}(M)$ for $M$ a finite free $W_n(K)$-module, as well as \emph{$\varphi$-biautomatic} functions $f: L_p \to \End_{\ZZ}(M)$.
We correspondingly define \emph{$p$-automatic} and \emph{$p$-quasi-automatic} elements of 
$W_n(K)((t^{\QQ}))$.
\end{defn}

We have the following replacement for \cite[Theorem~10]{k-series2}.
\begin{theorem} \label{T:mixed1}
Let $L$ be the integral closure of $K(t)$ in $K((t^{\QQ}))$. Then for $n$ a positive integer or $\infty$, the image of $W_n(L)$ in $W_n(K)((t^{\QQ})))^{\wedge}$ consists of the $p$-quasi-automatic elements.
\end{theorem}
\begin{proof}
It suffices to treat the case where $n$ is finite.
Note that \cite[Lemma~3.3.5]{k-series-automata} may be formally promoted to a corresponding statement with $K$ replaced by $W_n(K)$. Using this, we may
imitate the proof of Lemma~\ref{L:automatic to algebraic} to show that any $p$-quasi-automatic element of $W_n(K)((t^{\QQ}))$ belongs to the image of $W_n(L)$. To prove the reverse inclusion, we induct on $n$. The base case $n=1$ follows from Theorem~\ref{T:integral closure of rational field}. In general, if $x \in W_n(L)$, it follows that the image of $x$ in $L$ (via reduction mod $p$) is $p$-quasi-automatic. By lifting the terms in a presentation of $x$, we may construct $y \in W_n(K)((t^{\QQ}))$ which is $p$-quasi-automatic and congruent to $x$ modulo $p$.
Using the forward inclusion, we may write $x = y + pz$ for some $z \in W_{n-1}(L)$, and the induction hypothesis implies that $z$ is $p$-quasi-automatic. This completes the induction.
\end{proof}

By combining this with \cite[Theorem~7]{k-series2}, we formally deduce the following replacement for \cite[Theorem~11]{k-series2}.
\begin{defn}
As in \cite[\S 2]{k-series2}, we define $W(K)((p^{\QQ}))$ as the quotient of
$W(K)((t^{\QQ}))$ (or equivalently $W(K)((t^{\QQ}))^{\wedge}$) by the ideal consisting of those $x = \sum_i x_i t^i$ for which
$\sum_{n \in \ZZ} x_{n+i} p^n = 0$ for all $i \in \QQ$. By \cite[Proposition~2]{k-series2}, this is an algebraically closed field.
Each element of this field admits a unique lift to $W(K)((t^{\QQ}))$ of the form
$\sum_i [x_i] t^i$, where $x_i \in K$ and the brackets denote Teichm\"uller lifts.
\end{defn}

\begin{theorem} \label{T:mixed2}
Let $L$ be the completed integral closure of $K((t))$ in $K((t^{\QQ}))$. Then 
the completed integral closure of $W(K)[p^{-1}]$ in $W(K)((p^{\QQ}))$ is the completion of the image of the $p$-quasi-automatic elements of $W_n(K)((t^{\QQ}))^{\wedge}$.
Moreover, this image consists of those elements whose representation as
$\sum_i [x_i] t^i \in W(K)((t^{\QQ}))$ has the property that $\sum_i x_i t^i \in L$
(or equivalently by Theorem~\ref{T:main completed}, every truncation of $\sum_i x_i t^i$
is $p$-quasi-automatic).
\end{theorem}
\begin{proof}
The first statement is immediate from \cite[Theorem~7]{k-series2}. 
The second statement follows from the first statement and Theorem~\ref{T:mixed1} as in the
published proof of \cite[Theorem~11]{k-series2}.
\end{proof}

\begin{remark}
We take this opportunity to report some additional errata to \cite{k-series2}. Thanks to Chris Davis for providing some of these.
\begin{itemize}
\item
page 336: the equality in (2) should read
\[
c_n = f(1-b_1^{-1} - \cdots - b_{j-1} p^{-j+1} - p^{-n+l}(b_j p^{-j} + b_{j+1} p^{-j-1} + \cdots)).
\]
\item
page 377: after Theorem 10, ``we nos give'' should be ``we now give''.
\item
page 337: in point 2 above Theorem 11, the definition of $f$ should read $f_i = x_{(j+i)/a}$.
\end{itemize}
\end{remark}

\section{Computational aspects}
\label{sec:computational}

In \cite[\S 8.1]{k-series-automata}, some thoughts are recorded regarding the possible use of finite automata for making machine computations in the algebraic closure of $\FF_q(t)$, for $\FF_q$ a finite field. To our knowledge, no systematic effort has been made to realize these thoughts, with one exception: the very recent paper of Bridy \cite{bridy} gives explicit estimates for the state complexity of the automata required to represent algebraic power series. These bounds take the form $q^n$ where $n$ is a sum of certain geometric invariants.

However, we would like to take this opportunity to point out that recasting the construction in terms of $p$-composed functions may provide some algorithmic improvements. Namely, one of the potential difficulties with computing with automata is the opportunity for an explosion in the number of states required to describe a particular series; for example, given a language described by a finite automaton on $n$ states, its reversal may require as many as $2^n$ states. By contrast, reversing a $p$-composed function entails no explosion, as the underlying vector space is merely replaced with its dual space.
A related observation is that the key parameter when dealing with $p$-automatic series is the dimension of the vector space used to write down the corresponding $p$-composed function; by contrast, if one works in terms of finite automata, then it is the \emph{cardinality} of this vector space that controls the complexity. Since addition and multiplication of $p$-automatic series correspond to direct sum and tensor product of vector spaces, one expects much better behavior in the linear-algebraic approach; namely, by factoring dimensions into account, one might hope to obtain an analogue of Bridy's estimate in which the function $q^n$ is replaced by some polynomial in $n$.

For these reasons, it seems reasonable to try to use $p$-composed functions, rather than automata, as the basis for computing in an algebraic closure of $K(t)$ even in the case when $K$ is the algebraic closure of $\FF_p$. We leave this as a challenge for the interested reader.

\section{Effect on the literature}
\label{sec:effect}

We conclude by analyzing the effect of the error in \cite{k-series1}, and the corrections introduced in this paper, on the mathematical literature.
We begin with the effect on \cite{k-series1} itself.
\begin{itemize}
\item
The statements \cite[Theorem~1, Theorem~2]{k-series1} are taken from previous literature.
\item
The statements \cite[Lemma~3, Lemma~4, Corollary~5, Lemma~7, Lemma~14]{k-series1} are correct with their given proofs.
\item
The statements \cite[Theorem~6, Corollary~9, Corollary~11, Corollary~12, Theorem~15]{k-series1} follow respectively from Corollary~\ref{C:sabc}, Corollary~\ref{C:perfect coefficient subfield},
Theorem~\ref{T:main completed},
Corollary~\ref{C:full truncation}, Theorem~\ref{T:finite field2}.
In the same vein, \cite[Corollary~10]{k-series1} follows from Corollary~\ref{C:imperfect coefficient subfield} once its statement has been corrected to match that result.
\item
As noted earlier, \cite[Theorem~8]{k-series1} is true when $K = \overline{\FF}_p$ by Theorem~\ref{T:finite field}, but false otherwise by Example~\ref{exa:not TR}; the same is true of \cite[Corollary~11]{k-series1}.
\item
The statement \cite[Corollary~13]{k-series1} is false as written, but becomes true if one replaces ``twist-recurrent series'' by ``series algebraic over $K((t))$'' and may then be reinterpreted using Theorem~\ref{T:main}. See
Corollary~\ref{C:radius of convergence}.
\end{itemize}

We next consider the effect on \cite{k-series2}. As described in \S\ref{sec:mixed}, 
\cite[Theorem~10, Theorem~11]{k-series2} fail when $K \neq \overline{\FF}_p$,
and must be replaced by Theorem~\ref{T:mixed1} and Theorem~\ref{T:mixed2}, respectively.
None of the other results of \cite{k-series2} depend on \cite{k-series1}, so they remain unaffected.

We finally consider the effect on papers that cite \cite{k-series1} and/or \cite{k-series2}, as indicated by a MathSciNet search conducted in March 2016. Papers referencing \cite{k-series-automata} independently of \cite{k-series1, k-series2} are unaffected, and are thus not listed here.

\begin{itemize}
\item
The following papers make only passing references to \cite{k-series1} and/or \cite{k-series2}, with no reference to specific results, and are thus unaffected: \cite{abbes-hbaib}, \cite{abbes-hbaib-merkhi},
\cite{brzostowski}, 
\cite{brzostowzki-rodak},
\cite{cutkosky-kashcheyeva}, \cite{davis-kedlaya}, \cite{einsiedler-kapranov-lind},
\cite{firicel}, 
\cite{ghorbel-hbaib-zouari},
\cite{gluzman-yukalov},
\cite{hbaib-laabidi-merkhi},
\cite{hbaib-mahjoub-taktak},
\cite{jensen-markwig-markwig}, \cite{kedlaya-poonen}, \cite{leborgne}, \cite{papanikolas},
\cite{saturnino},
\cite{scheicher-sirvent}.

\item
The following papers only reference \cite{k-series1} (either explicitly or via \cite{k-series2}) in the case where $K = \overline{\FF}_p$,
and are thus unaffected: \cite{adamczewski-bell}, \cite{k-series-automata}, \cite{liu}, \cite{papanikolas}.

\item
The following papers only rely on the prior result \cite[Theorem~1]{k-series1}, and are thus unaffected:
\cite{bays-zilber}, 
\cite{chiarellotto-tsuzuki},
\cite{mosteig-sweedler}.

\item
The paper \cite{payne} references the transfinite Newton recurrence described in \cite[\S 2]{k-series2}, and is thus unaffected.

\item
The paper \cite{k-new-phigamma} references only the projection from generalized power series in mixed characteristic to positive characteristic, as described in \cite[Theorem~7]{k-series2}, and is thus unaffected.

\item
The paper \cite{mosteig} makes two references to \cite{k-series1}. One is in the proof of \cite[Proposition~5.4]{mosteig}, which depends only on the prior result \cite[Theorem~1]{k-series1} and is thus unaffected. The other is more serious: it is \cite[Proposition~5.2]{mosteig}, an explicit computation on twist-recurrent series which is then combined with \cite[Theorem~8]{k-series1} to deduce \cite[Proposition~5.3]{mosteig}.
Fortunately, the latter result holds for a simpler reason: for any homomorphism $\eta: \QQ \to K^\times$, the formula
\[
\sum_{i \in \QQ} x_i t^i \mapsto \sum_{i \in \QQ} \eta(i) x_i t^i
\]
defines an automorphism of $K((t^{\QQ}))$ acting on $K((t))$, which then also acts on the algebraic closure of $K((t))$ in $K((t^{\QQ}))$ (because any ring homomorphism preserves integral dependence). Consequently, \cite{mosteig} is ultimately unaffected.

\end{itemize}

\end{document}